\documentclass[english,a4paper]{amsart}
\usepackage[T1]{fontenc}
\usepackage[latin9]{inputenc}
\usepackage{amsmath}
\usepackage{amsthm}
\usepackage{amssymb}

\makeatletter
\theoremstyle{plain}
\newtheorem{thm}{\protect\theoremname}
  \theoremstyle{remark}
  \newtheorem{rem}[thm]{\protect\remarkname}
  \theoremstyle{plain}
  \newtheorem{lem}[thm]{\protect\lemmaname}
  \theoremstyle{plain}
  \newtheorem{prop}[thm]{\protect\propositionname}

\DeclareMathOperator{\ric}{Ric_{\eta}}

\makeatother

\usepackage{babel}
  \providecommand{\lemmaname}{Lemma}
  \providecommand{\propositionname}{Proposition}
  \providecommand{\remarkname}{Remark}
\providecommand{\theoremname}{Theorem}

\begin{document}

\title{Blowing up solutions for supercritical Yamabe problems on manifolds
with non umbilic boundary}
\thanks{The first authors was supported by Gruppo Nazionale per l'Analisi Matematica, la Probabilit\`{a} e le loro Applicazioni (GNAMPA) of Istituto Nazionale di Alta Matematica (INdAM) and by project PRA from Univeristy of Pisa}

\author{Marco G. Ghimenti}
\address{M. Ghimenti, \newline Dipartimento di Matematica Universit\`a di Pisa
Largo B. Pontecorvo 5, 56126 Pisa, Italy}
\email{marco.ghimenti@unipi.it}

\author{Anna Maria Micheletti}
\address{A. M. Micheletti, \newline Dipartimento di Matematica Universit\`a di Pisa
Largo B. Pontecorvo 5, 56126 Pisa, Italy}
\email{a.micheletti@dma.unipi.it.}

\begin{abstract}
We build blowing-up solutions for a supercritical perturbation of
the Yamabe problem on manifolds with boundary, provided the dimension
of the manifold is $n\ge7$ and the trace-free part of the second
fundamental form is non-zero everywhere on the boundary.
\end{abstract}


\maketitle

\section{Introduction}

In this paper we are interested in the existence of blowing-up solutions
to problems which are supercritical perturbation of the boundary Yamabe
problem, that is we are interested in finding a family $u_{\varepsilon}$
of solutions for the problem 
\begin{equation}
\left\{ \begin{array}{ll}
L_{g}u=0 & \text{ on }M\\
\frac{\partial}{\partial\nu}u+\frac{n-2}{2}h_{g}(x)u=(n-2)u^{\frac{n}{n-2}+\varepsilon} & \text{ on }\partial M
\end{array}\right.\label{eq:Pmain}
\end{equation}
where $\varepsilon>0$, $L_{g}:=\Delta_{g}-\frac{n-2}{4(n-1)}R_{g}$
is the conformal Laplacian, $R_{g}$ is the scalar curvature of $M$,
$h_{g}$ is the mean curvature on $\partial M$ and $\nu$ is the
outward normal.

Our main result is the following
\begin{thm}
\label{almaraz} Let $M$ be a manifold of positive type. Assume $n\ge7$,
and the trace-free second fundamental form of $\partial M$ is non
zero everywhere. Then there exists a solution $u_{\varepsilon}$ of
(\ref{eq:Pmain}) such that $u_{\varepsilon}$ blows up when $\varepsilon\rightarrow0^{+}$.
\end{thm}
This result can be read in a threefold way. At first, it is an existence
result for a supercritical Yamabe type problem in manifolds with boundary.
Secondly, it says that the family of solutions of this supercritical
problem is not a $C^{2}$ compact set, in fact it is not possible
find an uniform $C^{2}$ bound to the set $\left\{ u_{\varepsilon}\in H_{g}^{1}(M)\ :\ u_{\varepsilon}\text{ solution of (\ref{eq:Pmain})}\right\} _{\varepsilon}$.
This represents an obstrucion to the extension of the compactness
result of \cite{A3} to the supercritical case. Finally, Theorem \ref{almaraz}
has an interpretation in the sense of stabililty of Yamabe boundary
problem. Following Druet \cite{dru}, we say that the Yamabe boundary
problem 
\begin{equation}
\left\{ \begin{array}{ll}
L_{g}u=0 & \text{ on }M\\
\frac{\partial}{\partial\nu}u+\frac{n-2}{2}h_{g}(x)u=(n-2)u^{\frac{n}{n-2}} & \text{ on }\partial M
\end{array}\right.\label{eq:Pmain-1}
\end{equation}
is stable with respect to perturbation from above if, for any sequence
$\varepsilon_{n}\rightarrow0$, and any sequence of $u_{\varepsilon_{n}}$
solution of (\ref{eq:Pmain}) (with $\varepsilon_{n}$ as a parameter)
then $u_{\varepsilon_{n}}$ converges, up to subsequence, to a solution
$u_{0}$ of (\ref{eq:Pmain-1}) in $C^{2}(M)$. Since by our result
we showed a family of solutions which blows up while $\varepsilon\rightarrow0$,
then the Yamabe boundary problem is not stable from perturbation of
the critical exponent from above. Notice that in the same spirit the
result of Almaraz in \cite{A3} states also that the Yamabe boundary
problem is stable with respect of perturbation from below of the critical
exponent. 

We gives some recall of the classical Yamabe problem and to the Yamabe
boundary problem which allow us to give a framework to our result
and to clarifies the previous considerations.

\subsection{The Yamabe problem}

In 1960 Yamabe \cite{yam} raised the following question: Given $(M,g)$
a compact Riemannian manifold of dimension $n\ge3$, without boundary,
it is possible to find in the conformal class of $g$ a metric $\tilde{g}$
of constant scalar curvature?

Analitically this problem is equivalent to find positive solution
of the critical problem
\begin{equation}
-\Delta_{g}u+\frac{n-2}{4(n-1)}R_{g}u=cu^{\frac{n+2}{n-2}}\text{ in }M.\label{eq:Yam}
\end{equation}
Here $-\Delta_{g}$ is the Laplace Beltrami operator and $R_{g}$
is the scalar curvature of the original metric. In this case the new
metric $\tilde{g}:=u^{\frac{4}{n-2}}g$ has scalar curvature $R_{\tilde{g}}\equiv\frac{4c(n-1)}{n-2}$.
A positive answer to this problem was given by Yamabe \cite{yam},
Trudinger \cite{tru}, Aubin \cite{aub} and Schoen \cite{sch}. 

Problem (\ref{eq:Yam}) has a variational structure and a solution
could be find as a critical point of the quotient
\[
Q_{M}(u):=\frac{\int\limits _{M}\left(|\nabla u|^{2}+\frac{n-2}{4(n-1)}R_{g}u^{2}\right)d\mu_{g}}{\left(\int\limits _{M}|u|^{\frac{2n}{n-2}}d\sigma_{g}\right)^{\frac{n-2}{n}}},\ u\in H_{g}^{1}(M).
\]
It is well known that if 
\[
Q(M)=\inf_{u\in H_{g}^{1}(M)}Q_{M}(u)>0
\]
(otherwise said the manifold is of \emph{positive type}) the solution
of (\ref{eq:Yam}) is not unique, while uniqueness holds, up to symmetries,
if $Q(M)\le0$. At this point people start asking wheter if, given
a manifold of positive type, the family of solutions of (\ref{eq:Yam})
is $C^{2}$ compact, that is there is a $C^{2}$ uniform estimate
on the set of solutions. The question has been solved by S. Brendle,
M. A. Khuri, F. C. Marques and R. Schoen in a series of works \cite{bre,bre-mar,khu-mar-sch}.
Their result can be summarized as follows: compactness holds for dimensions
$3\le n\le24$ if the manifold is not conformally equivalent to a
round sphere, while for $n\ge25$ it is possible to construct manifolds
for which the set of solutions is not a compact set. 

In a series of paper Druet, Hebey and Robert (\cite{dru,dru-heb,DHR}
and the reference therein) studied the compactness of the solutions
of Yamabe type problems in which the linear term has the form $a(x)u$
where $a$ is a smooth function on $M$. As a further result they
proved the following stability result. Given a sequence a sequence
of smooth functions $a_{j}$ converging in $H^{2}(M)$ to $R_{g}$
for $j\rightarrow\infty$, with $a_{j}(x)\le R_{g}$ for all $x\in M$
and for all $j$, then the sequence of $\{u_{j}\}_{j}$ solutions
of 
\begin{equation}
-\Delta_{g}u+\frac{n-2}{4(n-1)}a_{j}u=cu^{\frac{n+2}{n-2}}\text{ in }M.\label{eq:Yam-1}
\end{equation}
converges, up to subsequences, to a solution $u_{\infty}$ of (\ref{eq:Yam})
while $j\rightarrow\infty$. 

In this sense they proved that the Yamabe problem (\ref{eq:Yam})
is stable with respect to the perturbations from below of the linear
term. Also, they showed blow up phenomena -and so instability- with
other choiches of perturbation of linear term. After them, the stability
vs. instability of (\ref{eq:Yam}) has been studied by several authors.
We limit ourselves to cite \cite{EPV,RV}, and \cite{DH}. In the
last paper a sharp description of the blow up profile for solutions
of perturbed problem is given. 

In a similar spirit Micheletti Pistoia and Vetois \cite{MPV} proved
blow up for a class of Yamabe type problems with slightly subcritical
and slightly supercritical nonlinearity.

\subsection{Boundary Yamabe problem}

In 1992 Escobar \cite{E92} generalized the classical Yamabe problem
to compact Riemannian manifolds with regular boundary. In this case
one can ask if there exists a conformal metric $\tilde{g}$ which
has both constant scalar curvature and constant mean curvature of
the boundary. In this case the analytic formulation for the Yamabe
problem takes the form
\begin{equation}
\left\{ \begin{array}{ll}
-\Delta_{g}u+\frac{n-2}{4(n-1)}R_{\varepsilon}u=c_{1}u^{\frac{n+2}{n-2}} & \text{ on }M\\
\frac{\partial}{\partial\nu}u+\frac{n-2}{2}h_{g}(x)u=c_{2}u^{\frac{n}{n-2}} & \text{ on }\partial M
\end{array}\right..\label{eq:YamB}
\end{equation}
Tipically people addressing to this problem fix one among the constant
$c_{1},c_{2}$ to be zero. In the case $c_{1}\neq0$ we limit ourselves
to cite -besides Escobar- Han and Li \cite{HL99}, Ambrosetti, Li
and Malchiodi \cite{ALM}, Djadli, Malchiodi, Ould Ahmedou \cite{DMO03,DMO04}
and the recent paper of Disconzi and Khuri which studied also compactness
of solutions for problem (\ref{eq:YamB}). The other case, in which
the target manifold is scalar flat while $c_{2}\neq0$, is interesting
since can be interpreted also as a multidimensional version of the
Riemann mapping problem. In addition, in this case the nonlinearity
appears in the boundary condition and it is critical for the Sobolev
immersion of $H^{1}(M)$ in $L^{p}(\partial M)$. The analytical version
of this problem is (\ref{eq:Pmain-1}) and its variational formulation
consists in finding critical points of the quotient 
\[
Q_{\partial}(u):=\frac{\int\limits _{M}\left(|\nabla u|^{2}+\frac{n-2}{4(n-1)}R_{g}u^{2}\right)d\mu_{g}+\int\limits _{\partial M}\frac{n-2}{2}H_{g}u^{2}d\sigma_{g}}{\left(\int\limits _{\partial M}|u|^{\frac{2(n-1)}{n-2}}d\sigma_{g}\right)^{\frac{n-2}{n-1}}},\ u\in H_{g}^{1}(M)
\]
 and in complete analogy with the Yamabe problem when 
\[
Q_{\partial}(M):=\inf\limits _{u\in H_{g}^{1}(M)}Q_{\partial}(u)
\]
is positive the solution is no more unique. For Problem (\ref{eq:Pmain-1})
the principal existence results are due to Escobar \cite{E92,E92-b,E96},
Marques \cite{Ma1,Ma2} and Almaraz \cite{A1}. Recenlty Brendle and
Chen \cite{BC} and Mayer and Ndiaye \cite{MN} covered all the remaining
cases. Concerning compactness of solutions in manifold of positive
type the first result for (\ref{eq:Pmain-1}) is given by Felli and
Ould Ahmedou which in \cite{FO03} proved compactness for scalar flat
manifolds not conformally equivalent to the closed disk, using the
Positive Mass Theorem. Compactness has been proved also for manifolds
whose trace-free part of the second fundamental form is non-zero everywhere
on the boundary by Almaraz \cite{A3} for dimensions $n\ge7$ and
recently extended to dimension $n=5,6$ by Kim, Musso and Wey \cite{KMW19},
and to dimensions $n=3$ \cite{ALM} and $n=4$ \cite{KMW19} without
any further assumption on the second fundamental form. When the trace-free
part of the second fundamental form is non-zero everywhere on the
boundary the boundary of the manifold is said \emph{non umbilic }and
Almaraz exploited this condition to bypass the Positive Mass Theorem.
This strategy has been recently adapted to manifold with \emph{umbilic}
boundary, that is when the tensor of the second funtamental form vanishes
on the boundary, by the authors, provided that the Weyl tensor is
always different from zero on the boundary and $n\ge6$ in \cite{GM20,GMsub}.

Dimension $n=24$ appears to be relevant also in boundary Yamabe problems,
in fact Almaraz in proved that for $n\ge25$ there exists manifold
with umbilic boundary for which compactness of solutions fails. Compactness
for manifold with umbilic boundary for $n\le24$ with no further assumption
is still an open problem. 

There is another strong analogy between classical and boundary Yamabe
problems about stability. In fact, in the aforementioned works of
Druet Hebey and Robert \cite{dru,dru-heb,DHR} Yamabe problem (\ref{eq:Yam})
appears to be stable for perturbation of the scalar curvature $R_{g}$
from below, while there are several examples of instability when $R_{\varepsilon}-R_{g}$
is somewhere positive. This phenomenon appears also in boundary Yamabe
problem when one try to perturb the mean curvature term: stabiliy
depend on the sign of $h_{\varepsilon}-h_{g}$. Indeed, in \cite{GMP18,GMP19}
the authors with Pistoia proved that the perturbed problem \textit{
\begin{equation}
\left\{ \begin{aligned} & L_{g}u=0\ \text{ in }M\\
 & \partial_{\nu}u+\frac{n-2}{2}h_{g}u+\varepsilon\gamma u=u^{\frac{2(n-1)}{n-2}-1}\text{ on }\partial M.
\end{aligned}
\right.\label{eq:Peps}
\end{equation}
}\textit{\emph{when $\gamma\in C^{2}(M)$ is strictly positive on
$\partial M$ then there is a family of solution which blows up as
$\varepsilon\rightarrow0$ when $n\ge7$ and the boundary is non umbilic
or $n\ge11$, the manifold has umbilic boundary and the Weyl tensor
never vanishes od $\partial M$ (this second case has been recently
extended up to dimension $8$ in \cite{GMwip}). In these two papers
the authors also construct examples of $\gamma$ which changes sign
for which the Yamabe boundary problem is not stable. On the other
hand, when $\gamma$ is everywhere negative on $\partial M$ the Yamabe
boundary problem appears to be stable for perturbation of mean curvature
as proved in \cite{GMdcds}, in the umbilic boundary case for $n\ge7$
and in non umbilic case when $n>8$ and the Weyl tensor does not vanish
on $\partial M$ and when $n=8$ with a slighter restrictive assumption
on the Weyl tensor, so the analogy with the role of scalar curvature
for classical Yamabe problem is complete. }}

In both classical and boundary Yamabe problems, as a corollary of
the compactness of solutions, people get that the problem is also
stable for perturbation from below of the critical exponent. So stability
is proved for scalar flat manifolds in \cite{FO03}, for manifolds
with non umbilic boundary in \cite{A3,KMW19}, and for umbilic boundary
manifolds whose Weyl tensor never vanishes on $\partial M$ in \cite{GM20,GMsub}.
On the other hand, in the present paper (and in \cite{GMwip} for
umbilic boundary manifolds) we prove that Yamabe boundary problem
is unstable with respect of perturbation form above of the critical
exponent.

Recently, in \cite{Sou}, The Yamabe type problem
\begin{equation}
\left\{ \begin{array}{ccc}
-\Delta_{g}u+A(x)u=0 &  & \text{on }M;\\
\frac{\partial u}{\partial\nu}-B(x)u=(n-2)\left(u^{+}\right)^{\frac{n}{n-2}} &  & \text{on \ensuremath{\partial}}M.
\end{array}\right.\label{eq:P-1}
\end{equation}
is studied and there are a series of compactness results which depend
on the sign of $A-R_{g}$ and $B-h_{g}$.

\section{Preliminaries. }

It is well known that there exists a global conformal transformation
which maps the manifold $M$ in a manifold for which the mean curvature
of the boundary is identically zero. In order to simplify our problem,
we choose a metric $(M,g)$ such that $h_{g}\equiv0$. We also set
$a=\frac{n-2}{4(n-1)}R_{g}$, so Problem (\ref{eq:Pmain}) reads as

\begin{equation}
\left\{ \begin{array}{ccc}
-\Delta_{g}u+au=0 &  & \text{on }M;\\
\frac{\partial u}{\partial\nu}=(n-2)\left(u^{+}\right)^{\frac{n}{n-2}+\varepsilon} &  & \text{on \ensuremath{\partial}}M.
\end{array}\right.\label{eq:P}
\end{equation}
Since the manifold is of positive type, then 
\[
\left\langle \left\langle u,v\right\rangle \right\rangle _{H}=\int_{M}(\nabla_{g}u\nabla_{g}v+auv)d\mu_{g}
\]
is an equivalent scalar product in $H_{g}^{1}$, which induces to
the equivalent norm $\|\cdot\|_{H}$. 

We define the exponent 
\[
s_{\varepsilon}=\frac{2(n-1)}{n-2}+n\varepsilon
\]
and the Banach space $\mathcal{H}:=H^{1}(M)\cap L^{s_{\varepsilon}}(\partial M)$
endowed with norm $\|u\|_{\mathcal{H}}=\|u\|_{H}+|u|_{L^{s_{\varepsilon}}(\partial M)}.$By
trace theorems, we have the following inclusion $W^{1,\tau}(M)\subset L^{t}(\partial M)$
for $t\le\tau\frac{n-1}{n-\tau}$. 

We recall the following result, by Nittka \cite[Th. 3.14]{Nit}
\begin{rem}
\label{rem:Nit}Let $\frac{2n}{n+2}\le q<\frac{n}{2}$, $r>0$. Then
there exists a constant $c$ such that if $f_{0}\in L^{q+r}(\Omega)$,
$\beta$ bounded and measurable and $g\in L^{\frac{(n-1)q}{n-q}+r}(\partial\Omega)$
and $u\in H^{1}(\Omega)$ is the unique weak solution of 
\[
\left\{ \begin{array}{ll}
Lu=f_{0} & \text{ on }\Omega\\
\frac{\partial}{\partial\nu}u+\beta u=g & \text{ on }\partial\Omega
\end{array}\right.
\]
where $L$ is a strictly elliptic second order operator, then 
\[
u\in L^{\frac{nq}{n-2q}}(\Omega),\ \left.u\right|_{\partial\Omega}L^{\frac{(n-1)q}{n-2q}}(\partial\Omega)\text{ and}
\]
\[
|u|_{L^{\frac{nq}{n-2q}}(\Omega)}+|u|_{L^{\frac{(n-1)q}{n-2q}}(\partial\Omega)}\le\left|f_{0}\right|_{L^{q+r}(\Omega)}+\left|g\right|_{L^{\frac{(n-1)q}{n-q}+r}(\partial\Omega)}
\]
\end{rem}
We consider $i:H^{1}(M)\rightarrow L^{\frac{2(n-1)}{n-2}}(\partial M)$
and its adjoint with respect to $\left\langle \left\langle \cdot,\cdot\right\rangle \right\rangle _{H}$
\[
i^{*}:L^{\frac{2(n-1)}{n}}(\partial M)\rightarrow H^{1}(M)
\]
defined by
\[
\left\langle \left\langle \varphi,i^{*}(g)\right\rangle \right\rangle _{H}=\int_{\partial M}\varphi gd\sigma\text{ for all }\varphi\in H^{1}
\]
so that $u=i^{*}(g)$ is the weak solution of the problem
\begin{equation}
\left\{ \begin{array}{ll}
-\Delta_{g}u+a(x)u=0 & \text{ on }M\\
\frac{\partial}{\partial\nu}u=g & \text{ on }\partial M
\end{array}\right..\label{eq:istellasopra}
\end{equation}
By \cite[Th. 3.14]{Nit} (see Remark \ref{rem:Nit}) we have that,
if $u\in H^{1}$ is a solution of (\ref{eq:istellasopra}), then for
$\frac{2n}{n+2}\le q<\frac{n}{2}$ and $r>0$ it holds
\begin{equation}
|u|_{L^{\frac{(n-1)q}{n-2q}}(\partial M)}=|i^{*}(g)|_{L^{\frac{(n-1)q}{n-2q}}(\partial M)}\le|g|_{L^{\frac{(n-1)q}{n-q}+r}(\partial M)}.\label{eq:nittka}
\end{equation}
By this result, we can choose $q,r$ such that
\begin{equation}
\frac{(n-1)q}{n-2q}=\frac{2(n-1)}{n-2}+n\varepsilon\text{ and }\frac{(n-1)q}{n-q}+r=\frac{2(n-1)+n(n-2)\varepsilon}{n+(n-2)\varepsilon}\label{eq:nittka1}
\end{equation}
that is 
\[
q=\frac{2n+n^{2}\left(\frac{n-2}{n-1}\right)\varepsilon}{n+2+2n\left(\frac{n-2}{n-1}\right)\varepsilon}\text{ and }r=\frac{2(n-1)+n(n-2)\varepsilon}{n+(n-2)\varepsilon}-\frac{2(n-1)+n(n-2)\varepsilon}{n+\left(n-2\right)\left(\frac{n}{n-1}\right)\varepsilon};
\]
so we have that, if $u\in L^{\frac{2(n-1)}{n-2}+n\varepsilon}(\partial M)$,
then $f_{\varepsilon}(u)\in L^{\frac{2(n-1)+n(n-2)\varepsilon}{n+\varepsilon(n-2)}}(\partial M)$
and, in light of (\ref{eq:nittka}), that also $i^{*}\left(f_{\varepsilon}(u)\right)\in L^{\frac{2(n-1)}{n-2}+n\varepsilon}(\partial M)$
. Here $f_{\varepsilon}(u)=(n-2)\left(u^{+}\right)^{\frac{n}{n-2}+\varepsilon}$. 

At this point we are allowed to rewrite Problem (\ref{eq:P}) as 
\begin{equation}
u=i^{*}\left(f_{\varepsilon}(u)\right),\ u\in\mathcal{H}.\label{eq:P*}
\end{equation}

We recall now some properties of the Fermi coordinates: since we choose
a conformal metric for which $H_{g}\equiv0$, we have the following
expansions in a neighborhood of $y=0$ (we use the notation $y=(z,t)$,
$z\in\mathbb{R}^{n}$ and $t\ge0$). We use the convention that $a,b,c,d=1,\dots,n$
and $i,j,k,l=1,\dots,n-1$ and the einstein convention on repeated
indices. 
\begin{align}
|g(y)|^{1/2}= & 1-\frac{1}{2}\left[\|\pi\|^{2}+\ric(0)\right]t^{2}-\frac{1}{6}\bar{R}_{ij}(0)z_{i}z_{j}+O(|y|^{3})\label{eq:|g|}\\
g^{ij}(y)= & \delta_{ij}+2h_{ij}(0)t+\frac{1}{3}\bar{R}_{ikjl}(0)z_{k}z_{l}+2\frac{\partial h_{ij}}{\partial z_{k}}(0)tz_{k}\nonumber \\
 & +\left[R_{injn}(0)+3h_{ik}(0)h_{kj}(0)\right]t^{2}+O(|y|^{3})\label{eq:gij}\\
g^{an}(y)= & \delta_{an}\label{eq:gin}
\end{align}
where $\pi$ is the second fundamental form and $h_{ij}(0)$ are its
coefficients, $\bar{R}_{ikjl}(0)$ and $R_{abcd}(0)$ are the curvature
tensor of $\partial M$ and $M$, respectively, $\bar{R}_{ij}(0)=\bar{R}_{ikjk}(0)$
are the coefficients of the Ricci tensor, and $\ric(0)=R_{nini}(0)=R_{nn}(0)$
(see \cite{E92}).

We conclude these preliminaries introducing the integral quantity
\[
I_{m}^{\alpha}=\int_{0}^{\infty}\frac{\rho^{\alpha}}{(1+\rho^{2})^{m}}d\rho
\]
and the following identities which are obtained by direct computation.

\begin{align}
I_{m}^{\alpha}=\frac{2m}{\alpha+1}I_{m+1}^{\alpha+2} & \text{ for }\alpha+1<2m\nonumber \\
I_{m}^{\alpha}=\frac{2m}{2m-\alpha-1}I_{m+1}^{\alpha} & \text{ for }\alpha+1<2m\label{eq:Iam}\\
I_{m}^{\alpha}=\frac{2m-\alpha-3}{\alpha+1}I_{m}^{\alpha+2} & \text{ for }\alpha+3<2m.\nonumber 
\end{align}

\section{The finite dimensional reduction.\label{sec:reduction}}

Given $q\in\partial M$ and $\psi_{q}^{\partial}:\mathbb{R}_{+}^{n}\rightarrow M$
the Fermi coordinates in a neighborhood of $q$, we define 
\begin{align*}
W_{\delta,q}(\xi) & =U_{\delta}\left(\left(\psi_{q}^{\partial}\right)^{-1}(\xi)\right)\chi\left(\left(\psi_{q}^{\partial}\right)^{-1}(\xi)\right)=\\
 & =\frac{1}{\delta^{\frac{n-2}{2}}}U\left(\frac{y}{\delta}\right)\chi(y)=\frac{1}{\delta^{\frac{n-2}{2}}}U\left(x\right)\chi(\delta x)
\end{align*}
where $y=(z,t)$, with $z\in\mathbb{R}^{n-1}$ and $t\ge0$, $\delta x=y=\left(\psi_{q}^{\partial}\right)^{-1}(\xi)$
and $\chi$ is a radial cut off function, with support in ball of
radius $R$.

Here $U_{\delta}(y)=\frac{1}{\delta^{\frac{n-2}{2}}}U\left(\frac{y}{\delta}\right)$
is the one parameter family of solution of the problem 
\begin{equation}
\left\{ \begin{array}{ccc}
-\Delta U_{\delta}=0 &  & \text{on }\mathbb{R}_{+}^{n};\\
\frac{\partial U_{\delta}}{\partial t}=-(n-2)U_{\delta}^{\frac{n}{n-2}} &  & \text{on \ensuremath{\partial}}\mathbb{R}_{+}^{n}.
\end{array}\right.\label{eq:Udelta}
\end{equation}
and ${\displaystyle U(z,t):=\frac{1}{\left[(1+t)^{2}+|z|^{2}\right]^{\frac{n-2}{2}}}}$
is the standard bubble in $\mathbb{R}_{+}^{n}$.

Now, let us consider the linearized problem 
\begin{equation}
\left\{ \begin{array}{ccc}
 & -\Delta\phi=0 & \text{on }\mathbb{R}_{+}^{n},\\
 & \frac{\partial\phi}{\partial t}+nU^{\frac{2}{n-2}}\phi=0 & \text{on \ensuremath{\partial}}\mathbb{R}_{+}^{n},\\
 & \phi\in H^{1}(\mathbb{R}_{+}^{n}).
\end{array}\right.\label{eq:linearizzato}
\end{equation}
and it is well know (see, for instance, \cite[Lemma 6]{GMP18}) that
 every solution of (\ref{eq:linearizzato}) is a linear combination
of the functions $j_{1},\dots,j_{n}$ defined by . 
\begin{eqnarray}
j_{i}=\frac{\partial U}{\partial x_{i}},\ i=1,\dots n-1 &  & j_{n}=\frac{n-2}{2}U+\sum_{i=1}^{n}y_{i}\frac{\partial U}{\partial y_{i}}.\label{eq:sol-linearizzato}
\end{eqnarray}

Given $q\in\partial M$ we define, for $b=1,\dots,n$ 
\[
Z_{\delta,q}^{b}(\xi)=\frac{1}{\delta^{\frac{n-2}{2}}}j_{b}\left(\frac{1}{\delta}\left(\psi_{q}^{\partial}\right)^{-1}(\xi)\right)\chi\left(\left(\psi_{q}^{\partial}\right)^{-1}(\xi)\right)
\]
and we decompose $H^{1}(M)$ in the direct sum of the following two
subspaces 
\begin{align*}
K_{\delta,q} & =\text{Span}\left\langle Z_{\delta,q}^{1},\dots,Z_{\delta,q}^{n}\right\rangle \\
K_{\delta,q}^{\bot} & =\left\{ \varphi\in H^{1}(M)\ :\ \left\langle \left\langle \varphi,Z_{\delta,q}^{b}\right\rangle \right\rangle _{H}=0,\ b=1,\dots,n\right\} 
\end{align*}
and we define the projections 
\begin{eqnarray*}
\Pi=H^{1}(M)\rightarrow K_{\delta,q} &  & \Pi^{\bot}=H^{1}(M)\rightarrow K_{\delta,q}^{\bot}.
\end{eqnarray*}

Given $q\in\partial M$ we also define in a similar way 
\[
V_{\delta,q}(\xi)=\frac{1}{\delta^{\frac{n-2}{2}}}v_{q}\left(\frac{1}{\delta}\left(\psi_{q}^{\partial}\right)^{-1}(\xi)\right)\chi\left(\left(\psi_{q}^{\partial}\right)^{-1}(\xi)\right),
\]
and 
\begin{equation}
\left(v_{q}\right)_{\delta}(y)=\frac{1}{\delta^{\frac{n-2}{2}}}v_{q}\left(\frac{y}{\delta}\right);\label{eq:vqdelta}
\end{equation}
here $v_{q}:\mathbb{R}_{+}^{n}\rightarrow\mathbb{R}$ is the unique
solution of the problem 
\begin{equation}
\left\{ \begin{array}{ccc}
-\Delta v=2h_{ij}(q)t\partial_{ij}^{2}U &  & \text{on }\mathbb{R}_{+}^{n};\\
\frac{\partial v}{\partial t}+nU^{\frac{2}{n-2}}v=0 &  & \text{on \ensuremath{\partial}}\mathbb{R}_{+}^{n}.
\end{array}\right.\label{eq:vqdef}
\end{equation}
such that $v_{q}$ is $L^{2}(\mathbb{R}_{+}^{n})$-ortogonal to $j_{b}$
for all $b=1,\dots,n$ Here $h_{ij}$ is the second fundamental form
and we use the Einstein convention of repeated indices. We remark
that 
\begin{equation}
|\nabla^{r}v_{q}(y)|\le C(1+|y|)^{3-r-n}\text{ for }r=0,1,2,\label{eq:gradvq}
\end{equation}
\begin{equation}
\int_{\partial\mathbb{R}_{+}^{n}}U^{\frac{n}{n-2}}v_{q}=0\label{eq:Uvq}
\end{equation}
and 
\begin{equation}
\int_{\partial\mathbb{R}_{+}^{n}}\Delta v_{q}v_{q}dzdt\le0.\label{new}
\end{equation}
In addition, the map $q\mapsto v_{q}$ is in $C^{2}(\partial M)$.
(see \cite[Proposition 5.1 and estimate (5.9)]{A3} and \cite{GMP18}
for the last claim).

The function $v_{q}$ is related to the first order expansion of the
metric tensor $g_{ij}$ (see eq. (\ref{eq:gij})) and provides a sharp
correction term of the bubble in order to give a good ansatz for a
solution. Indeed, recasting Problem (\ref{eq:P}) as $u=i^{*}\left(f_{\varepsilon}(u)\right),\ u\in\mathcal{H}$,
we look for solution of (\ref{eq:P*}) having the form 

\[
u=W_{\delta,q}+\delta V_{\delta,q}+\phi,\text{ with }\phi\in K_{\delta,q}^{\bot}\cap\mathcal{H}.
\]
or, in other terms, we want to solve the following couple of equation
\begin{align}
\Pi^{\bot}\left\{ W_{\delta,q}+\delta V_{\delta,q}+\phi-i^{*}\left(f_{\varepsilon}(W_{\delta,q}+\delta V_{\delta,q}+\phi)\right)\right\}  & =0;\label{eq:P-Kort}\\
\Pi\left\{ W_{\delta,q}+\delta V_{\delta,q}+\phi-i^{*}\left(f_{\varepsilon}(W_{\delta,q}+\delta V_{\delta,q}+\phi)\right)\right\}  & =0.\label{eq:P-K}
\end{align}
We rewrite (\ref{eq:P-Kort}) as

\[
L(\phi)=N(\phi)+R
\]
where $L:=L_{\delta,q}$ is the linear operator 
\begin{align}
L: & K_{\delta,q}^{\bot}\cap\mathcal{H}\rightarrow K_{\delta,q}^{\bot}\cap\mathcal{H}\nonumber \\
L(\phi):= & \Pi^{\bot}\left\{ \phi-i^{*}\left(f'_{\varepsilon}(W_{\delta,q}+\delta V_{\delta,q})[\phi]\right)\right\} \label{eq:defL}
\end{align}
and the nonlinear term $N(\phi)$ and the remainder term $R$ are
\begin{align}
N(\phi):= & \Pi^{\bot}\left\{ i^{*}\left(f(W_{\delta,q}+\delta V_{\delta,q}+\phi)-f(W_{\delta,q}+\delta V_{\delta,q})-f'(W_{\delta,q}+\delta V_{\delta,q})[\phi]\right)\right\} ;\label{eq:defN}\\
R:= & \Pi^{\bot}\left\{ i^{*}\left(f(W_{\delta,q}+\delta V_{\delta,q})\right)-W_{\delta,q}-\delta V_{\delta,q}\right\} .\label{eq:defR}
\end{align}
The rest of this section is devoted to show that for any choice of
$\delta,q$ a solution $\phi$ of the (\ref{eq:P-Kort}) exists. We
remark that the choice of $v_{q}$ is crucial to obtain a good estimate
on the size of the remainder term $R$, which allows us to prove the
main result of this section, Proposition \ref{prop:phi}.
\begin{lem}
\label{lem:L}Assume $n\ge7$ and let $\delta=\sqrt{\varepsilon}\lambda$
For $a,b\in\mathbb{R}$, $0<a<b$ there exists a positive constant
$C_{0}=C_{0}(a,b)$ such that, for $\varepsilon$ small, for any $q\in\partial M$,
for any $\lambda\in[a,b]$ and for any $\phi\in K_{\delta,q}^{\bot}\cap\mathcal{H}$
there holds
\[
\|L_{\delta,q}(\phi)\|_{\mathcal{H}}\ge C_{0}\|\phi\|_{\mathcal{H}}.
\]
\end{lem}
\begin{proof}
The proof of this Lemma is very similar to the proof of \cite[Lemma 2]{GMP16}
and will be omitted.
\end{proof}
\begin{lem}
\label{lem:R}Assume $n\ge7$ and let $\delta=\sqrt{\varepsilon}\lambda$
For $a,b\in\mathbb{R}$, $0<a<b$ there exists a positive constant
$C_{1}=C_{1}(a,b)$ such that, for $\varepsilon$ small, for any $q\in\partial M$
and for any $\lambda\in[a,b]$ there holds
\[
\|R\|_{\mathcal{H}}\le C_{1}\varepsilon\left|\ln\varepsilon\right|
\]
\end{lem}
\begin{proof}
We estimate firstly $\|R\|_{H}$. We have 
\begin{align*}
\left\Vert R\right\Vert _{H} & \le\left\Vert i^{*}\left(f_{\varepsilon}(W_{\delta,q}+\delta V_{\delta,q})\right)-i^{*}\left(f_{0}(W_{\delta,q}+\delta V_{\delta,q})\right)\right\Vert _{H}\\
 & +\left\Vert i^{*}\left(f_{0}(W_{\delta,q}+\delta V_{\delta,q})\right)-W_{\delta,q}-\delta V_{\delta,q}\right\Vert _{H}.
\end{align*}
We start by estimating the second term. By definiton of $i^{*}$ there
exists $\Gamma=i^{*}\left(f_{0}(W_{\delta,q}+\delta V_{\delta,q})\right)$,
that is a function $\Gamma$ solving
\begin{equation}
\left\{ \begin{array}{ll}
-\Delta_{g}\Gamma+a(x)\Gamma=0 & \text{ on }M\\
\frac{\partial}{\partial\nu}\Gamma=f_{0}(W_{\delta,q}+\delta V_{\delta,q}) & \text{ on }\partial M
\end{array}\right..\label{eq:gamma}
\end{equation}
So we have
\begin{multline*}
\left\Vert i^{*}\left(f_{0}(W_{\delta,q}+\delta V_{\delta,q}\right)-W_{\delta,q}-\delta V_{\delta,q}\right\Vert _{H}^{2}=\|\Gamma-W_{\delta,q}-\delta V_{\delta,q}\|_{H}^{2}\\
=\int_{M}\left[-\Delta_{g}(\Gamma-W_{\delta,q}-\delta V_{\delta,q})+a(\Gamma-W_{\delta,q}-\delta V_{\delta,q})\right](\Gamma-W_{\delta,q}-\delta V_{\delta,q})d\mu_{g}\\
+\int_{\partial M}\left[\frac{\partial}{\partial\nu}(\Gamma-W_{\delta,q}-\delta V_{\delta,q})\right](\Gamma-W_{\delta,q}-\delta V_{\delta,q})d\sigma\\
=\int_{M}\left[\Delta_{g}(W_{\delta,q}+\delta V_{\delta,q})-a(W_{\delta,q}+\delta V_{\delta,q})\right](\Gamma-W_{\delta,q}-\delta V_{\delta,q})d\mu_{g}\\
+\int_{\partial M}\left[f_{0}(W_{\delta,q}+\delta V_{\delta,q})-\frac{\partial}{\partial\nu}(W_{\delta,q}+\delta V_{\delta,q})\right](\Gamma-W_{\delta,q}-\delta V_{\delta,q})d\sigma\\
=:I_{1}+I_{2}.
\end{multline*}
We have

\begin{align*}
I_{1} & \le C\left|\Delta_{g}(W_{\delta,q}+\delta V_{\delta,q})-a(W_{\delta,q}+\delta V_{\delta,q})\right|_{L^{\frac{2n}{n+2}}(M)}\left\Vert \Gamma-W_{\delta,q}-\delta V_{\delta,q}\right\Vert _{H}.
\end{align*}
By direct computation we have immediately that $\left|(W_{\delta,q}+\delta V_{\delta,q})\right|_{L^{\frac{2n}{n+2}}}=O(\delta^{2})$.
Then we proceed as in \cite[eq. (35)]{GMP18}. Recalling that in local
charts the Laplace Beltrami operator is 
\begin{eqnarray*}
\Delta_{g}W_{\delta,q} & = & \Delta_{\text{euc}}\left(U_{\delta}(u)\chi(y)\right)+[g^{ij}(y)-\delta_{ij}]\partial_{ij}^{2}\left(U_{\delta}(u)\chi(y)\right)\\
 &  & -g^{ij}(y)\Gamma_{ij}^{k}(y)\partial_{k}\left(U_{\delta}(u)\chi(y)\right)
\end{eqnarray*}
and noticing that by (\ref{eq:|g|}) and (\ref{eq:gij}) for the Christoffel
symbols holds $\Gamma_{ij}^{k}(y)=O(|y|)$, we have, in variables
$y=\delta x$, 
\begin{eqnarray}
\Delta_{g}W_{\delta,q} & = & U_{\delta}(u)\Delta_{\text{euc}}\left(\chi(y)\right)+2\nabla U_{\delta}(u)\nabla\chi(y)\nonumber \\
 &  & +[g^{ij}(y)-\delta_{ij}]\partial_{ij}^{2}\left(U_{\delta}(u)\chi(y)\right)-g^{ij}(y)\Gamma_{ij}^{k}(y)\partial_{k}\left(U_{\delta}(u)\chi(y)\right)\nonumber \\
 & = & \frac{1}{\delta^{\frac{n-2}{2}}}\left(2h_{ij}(0)\delta x_{n}\frac{1}{\delta^{2}}\partial_{ij}U(x)+g^{ij}(x)\Gamma_{ij}^{k}(x)\frac{1}{\delta}\partial_{k}U+o(\delta)c(x)\right)\nonumber \\
 & = & \frac{1}{\delta^{\frac{n}{2}}}\left(2h_{ij}(0)x_{n}\partial_{ij}^{2}U(x)+O(\delta)c(x)\right)\label{eq:R1}
\end{eqnarray}
where, with abuse of notation, we call $c(x)$ function with$\left|\int_{\mathbb{R}^{n}}c(x)dx\right|\le C$.
In a similar way, by (\ref{eq:vqdef}) and by (\ref{eq:gij}) we have
\begin{equation}
\delta\Delta_{g}V_{\delta,q}=\frac{1}{\delta^{\frac{n}{2}}}\left(-2h_{ij}(0)x_{n}\partial_{ij}^{2}U(y)+O(\delta)c(y)\right)\label{eq:R2}
\end{equation}
Thus, in local chart by (\ref{eq:R1}) and (\ref{eq:R2}) we get 
\begin{equation}
|\Delta_{g}(W_{\delta,q}+\delta V_{\delta,q})|_{L^{\frac{2n}{n+2}}(M)}=\delta^{n\frac{n+2}{2n}}\frac{1}{\delta^{\frac{n}{2}}}O(\delta)=O(\delta^{2})\label{eq:deltaw+v}
\end{equation}
and we conclude that 
\begin{equation}
I_{1}=O(\delta^{2})\left\Vert \Gamma-W_{\delta,q}\right\Vert _{H}.\label{eq:I1}
\end{equation}
Notice that the estimate (\ref{eq:deltaw+v}) is possible since we
carefully choose the function $v_{q}$ as a solution of (\ref{eq:vqdef}).

For the second integral $I_{2}$ we proceed in a similar way, getting
\begin{align}
I_{2} & \le C\left|f_{0}(W_{\delta,q}+\delta V_{\delta,q})-\frac{\partial}{\partial\nu}(W_{\delta,q}-\delta V_{\delta,q})\right|_{L^{\frac{2(n-1)}{n}}(\partial M)}\left\Vert \Gamma-W_{\delta,q}-\delta V_{\delta,q}\right\Vert _{H}\label{eq:I2}
\end{align}
and again, arguing similarly to \cite[Lemma 9]{GMP18}, we have, since
$U$ is a solution of (\ref{eq:Udelta}),
\begin{align*}
\int_{\partial M}\left[(n-2)W_{\delta,q}^{\frac{n}{n-2}}-\frac{\partial}{\partial\nu}W_{\delta,q}\right]R & \le\left|(n-2)W_{\delta,q}^{\frac{n}{n-2}}-\frac{\partial}{\partial\nu}W_{\delta,q}\right|_{L^{\frac{2(n-1)}{n}}(\partial M)}\|R\|_{H}\\
 & =O(\delta^{2})\|R\|_{H}.
\end{align*}
Now we estimate 
\begin{multline*}
\int_{\partial M}\left\{ (n-2)\left[\left((W_{\delta,q}+\delta V_{\delta,q})^{+}\right)^{\frac{n}{n-2}}-W_{\delta,q}^{\frac{n}{n-2}}\right]-\delta\frac{\partial V_{\delta,q}}{\partial\nu}\right\} Rd\sigma\\
\le c\left|(n-2)\left[\left((W_{\delta,q}+\delta V_{\delta,q})^{+}\right)^{\frac{n}{n-2}}-W_{\delta,q}^{\frac{n}{n-2}}\right]-\delta\frac{\partial V_{\delta,q}}{\partial\nu}\right|_{L^{\frac{2(n-1)}{n}}(\partial M)}\|R\|_{H}
\end{multline*}
and, by Taylor expansion and by definition of the function $v_{q}$
(see (\ref{eq:vqdef}) ) 
\begin{multline*}
\left|(n-2)\left[\left((W_{\delta,q}+\delta V_{\delta,q})^{+}\right)^{\frac{n}{n-2}}-W_{\delta,q}^{\frac{n}{n-2}}\right]-\delta\frac{\partial V_{\delta,q}}{\partial\nu}\right|_{L^{\frac{2(n-1)}{n}}(\partial M)}\\
\le\left(\int_{|z|<\frac{1}{\delta}}\left((n-2)\left[\left((U+\delta v_{q})^{+}\right)^{\frac{n}{n-2}}-U^{\frac{n}{n-2}}\right]+\delta\frac{\partial v_{q}}{\partial t}\right)^{\frac{2(n-1)}{n}}\right)^{\frac{n}{2(n-1)}}+o(\delta^{2})\\
\le\delta\left(\int_{|z|<\frac{1}{\delta}}\left(n\left((U+\theta\delta v_{q})^{+}\right)^{\frac{2}{n-2}}v_{q}-U^{\frac{2}{n-2}}v_{q}\right)^{\frac{2(n-1)}{n}}\right)^{\frac{n}{2(n-1)}}+o(\delta^{2})
\end{multline*}
Notice that, for $\delta$ small enough, $U+\theta\delta v_{q}>0$
if $|y|\le1/\delta$ by the decay estimates (\ref{eq:gradvq}). At
this point, using again Taylor expansion, we have 
\begin{multline*}
\int_{|z|<\frac{1}{\delta}}\left[\left|\left((U+\theta\delta v_{q})^{+}\right)^{\frac{2}{n-2}}-U^{\frac{2}{n-2}}\right||v_{q}|\right]^{\frac{2(n-1)}{n}}\\
=\delta^{\frac{2(n-1)}{n}}\int_{|z|<\frac{1}{\delta}}\left(U+\theta_{1}\delta v_{q}\right)^{\frac{-2(n-1)(n-4)}{n(n-2)}}|v_{q}|^{\frac{4(n-1)}{n}}dz
\end{multline*}
where the last integral is bounded since $n\ge7$. 

Thus $\left|(n-2)\left[\left((W_{\delta,q}+\delta V_{\delta,q})^{+}\right)^{\frac{n}{n-2}}-W_{\delta,q}^{\frac{n}{n-2}}\right]-\delta\frac{\partial V_{\delta,q}}{\partial\nu}\right|_{L^{\frac{2(n-1)}{n}}(\partial M)}=O(\delta^{2})$,
so $I_{2}=O(\delta^{2})\left\Vert \Gamma-W_{\delta,q}\right\Vert _{H}$
and, consequentely,
\[
\left\Vert i^{*}\left(f_{0}(W_{\delta,q}+\delta V_{\delta,q}\right)-W_{\delta,q}-\delta V_{\delta,q}\right\Vert _{H}=O(\delta^{2}).
\]
To conclude the first part of the proof we estimate the term 
\[
\left\Vert i^{*}\left(f_{\varepsilon}(W_{\delta,q}+\delta V_{\delta,q})\right)-i^{*}\left(f_{0}(W_{\delta,q}+\delta V_{\delta,q})\right)\right\Vert _{H}.
\]
It is useful to recall the following Taylor expansions with respect
to $\varepsilon$
\begin{align}
U^{\varepsilon} & =1+\varepsilon\ln U+\frac{1}{2}\varepsilon^{2}\ln^{2}U+o(\varepsilon^{2})\label{eq:Uallaeps}\\
\delta^{-\varepsilon\frac{n-2}{2}} & =1-\varepsilon\frac{n-2}{2}\ln\delta+\varepsilon^{2}\frac{(n-2)^{2}}{8}\ln^{2}\delta+o(\varepsilon^{2}\ln^{2}\delta)\label{eq:deltaallaeps}
\end{align}
We have that 
\begin{multline}
\left\Vert i^{*}\left(f_{\varepsilon}(W_{\delta,q}+\delta V_{\delta,q})\right)-i^{*}\left(f_{0}(W_{\delta,q}+\delta V_{\delta,q})\right)\right\Vert _{H}\\
\le\left|\left(W_{\delta,q}+\delta V_{\delta,q}\right)^{\frac{n}{n-2}+\varepsilon}-\left(W_{\delta,q}+\delta V_{\delta,q}\right)^{\frac{n}{n-2}}\right|_{L^{\frac{2(n-1)}{n}}(\partial M)}\\
=\left\{ \int_{|z|<\frac{1}{\delta}}\left[\left(\frac{1}{\delta^{\varepsilon\frac{n-2}{2}}}(U+\delta v_{q})^{\varepsilon}-1\right)(U+\delta v_{q})^{\frac{n}{n-2}}\right]^{\frac{2(n-1)}{n}}dz\right\} ^{\frac{n}{2(n-1)}}+O(\delta^{2})\\
\le\left\{ \int_{|z|<\frac{1}{\delta}}\left|\left(-\frac{n-2}{2}\varepsilon\ln\delta+\varepsilon\ln(U+\delta v_{q})+O(\varepsilon^{2}\ln\delta)\right)U^{\frac{n}{n-2}}\right|^{\frac{2(n-1)}{n}}dz\right\} ^{\frac{n}{2(n-1)}}+O(\delta^{2})\\
=O(\varepsilon|\ln\delta|)+O(\varepsilon)+O(\delta^{2})\label{eq:feps-f0}
\end{multline}
Remembering that $\delta=\lambda\sqrt{\varepsilon}$ we get the required
estimate in $H$-norm.

To conclude the proof, we have to control $|R|_{L^{s_{\varepsilon}}(\partial M)}$.
As in the previous case we consider 
\begin{align*}
|R|_{L^{s_{\varepsilon}}(\partial M)}\le & \left|i^{*}\left(f_{\varepsilon}(W_{\delta,q}+\delta V_{\delta,q})\right)-i^{*}\left(f_{0}(W_{\delta,q}+\delta V_{\delta,q})\right)\right|_{L^{s_{\varepsilon}}(\partial M)}\\
 & +\left|i^{*}\left(f_{0}(W_{\delta,q}+\delta V_{\delta,q})\right)-W_{\delta,q}-\delta V_{\delta,q}\right|_{L^{s_{\varepsilon}}(\partial M)}
\end{align*}
and we start estimating the second term. Taken $\Gamma=i^{*}\left(f_{0}(W_{\delta,q}+\delta V_{\delta,q})\right)$
the solution of (\ref{eq:gamma}), we have that the function $\Gamma-W_{\delta,q}-\delta V_{\delta,q}$
solves the problem
\[
\left\{ \begin{array}{ll}
-\Delta_{g}(\Gamma-W_{\delta,q}-\delta V_{\delta,q})+a(x)(\Gamma-W_{\delta,q}-\delta V_{\delta,q})\\
=-\Delta_{g}(W_{\delta,q}+\delta V_{\delta,q})+a(x)(W_{\delta,q}+\delta V_{\delta,q}) & \text{ on }M\\
\\
\frac{\partial}{\partial\nu}(\Gamma-W_{\delta,q}-\delta V_{\delta,q})=f_{0}(W_{\delta,q}+\delta V_{\delta,q})-\frac{\partial}{\partial\nu}(W_{\delta,q}+\delta V_{\delta,q}) & \text{ on }\partial M
\end{array}\right..
\]
We choose $q=\frac{2n+n^{2}\left(\frac{n-2}{n-1}\right)\varepsilon}{n+2+2n\left(\frac{n-2}{n-1}\right)\varepsilon}$
and $r=\varepsilon$, so, by Remark \ref{rem:Nit}, we get 
\begin{align*}
|\Gamma-W_{\delta,q}-\delta V_{\delta,q}|_{L^{s_{\varepsilon}}(\partial M)}\le & |-\Delta_{g}(W_{\delta,q}+\delta V_{\delta,q})+a(x)(W_{\delta,q}+\delta V_{\delta,q})|_{L^{q+\varepsilon}(M)}\\
 & +\left|f_{0}(W_{\delta,q}+\delta V_{\delta,q})-\frac{\partial}{\partial\nu}(W_{\delta,q}+\delta V_{\delta,q})\right|_{L^{\frac{(n-1)q}{n-q}+\varepsilon}(\partial M)}.
\end{align*}
We remark that with our choice we can write $q=\frac{2n}{n+2}+O^{+}(\varepsilon)$,
$\frac{1}{q+\varepsilon}=\frac{n+2}{2n}-O^{+}(\varepsilon)$ and $\frac{(n-1)q}{n-q}+\varepsilon=\frac{2(n-1)}{n}+O^{+}(\varepsilon)$
where $0<O^{+}(\varepsilon)<C\varepsilon$ for some positive constant
$C$. 

Easily we get
\[
|a(x)(W_{\delta,q}+\delta V_{\delta,q})|_{L^{q+\varepsilon}(M)}=O\left(\delta^{2-O^{+}(\varepsilon)}\right).
\]

Moreover, proceeding as in the first part of the proof we get
\begin{align*}
|-\Delta_{g}(W_{\delta,q}+\delta V_{\delta,q})|_{L^{q+\varepsilon}(M)} & =O\left(\delta^{2-O^{+}(\varepsilon)}\right);\\
\left|f_{0}(W_{\delta,q}+\delta V_{\delta,q})-\frac{\partial}{\partial\nu}(W_{\delta,q}+\delta V_{\delta,q})\right|_{L^{\frac{(n-1)q}{n-q}+\varepsilon}(\partial M)} & =O\left(\delta^{2-O^{+}(\varepsilon)}\right).
\end{align*}
To complete the proof we estimate 
\[
\left|i^{*}\left(f_{\varepsilon}(W_{\delta,q}+\delta V_{\delta,q})\right)-i^{*}\left(f_{0}(W_{\delta,q}+\delta V_{\delta,q})\right)\right|_{L^{s_{\varepsilon}}(\partial M)}.
\]
Again we have
\begin{multline*}
\left|i^{*}\left(f_{\varepsilon}(W_{\delta,q}+\delta V_{\delta,q})\right)-i^{*}\left(f_{0}(W_{\delta,q}+\delta V_{\delta,q})\right)\right|_{L^{s_{\varepsilon}}(\partial M)}\\
\le\left|f_{\varepsilon}(W_{\delta,q}+\delta V_{\delta,q})-f_{0}(W_{\delta,q}+\delta V_{\delta,q})\right|_{L^{\frac{2(n-1)}{n}+O^{+}(\varepsilon)}(\partial M)}\\
\le\delta^{-O^{+}(\varepsilon)}\left\{ \int_{|z|<\frac{1}{\delta}}\left[\left(\frac{1}{\delta^{\varepsilon\frac{n-2}{2}}}(U+\delta v_{q})^{\varepsilon}-1\right)U^{\frac{n}{n-2}}(z,0)\right]^{\frac{2(n-1)}{n}+O^{+}(\varepsilon)}dz\right\} ^{\frac{1}{\frac{n}{2(n-1)}+O^{+}(\varepsilon)}}\\
+O(\delta^{2-O^{+}(\varepsilon)})\\
=\delta^{-O^{+}(\varepsilon)}\left\{ O(\varepsilon\left|\ln\delta\right|)+O(\varepsilon)+O(\delta^{2})\right\} .
\end{multline*}
and, since $\delta=\lambda\sqrt{\varepsilon}$, we have $\delta^{-O^{+}(\varepsilon)}=O(1)$
and we complete the proof.
\end{proof}
\begin{prop}
\label{prop:phi}Assume $n\ge7$ and let $\delta=\lambda\sqrt{\varepsilon}$
For $a,b\in\mathbb{R}$, $0<a<b$ there exists a positive constant
$C=C(a,b)$ such that, for $\varepsilon$ small, for any $q\in\partial M$
and for any $\lambda\in[a,b]$ there exists a unique $\phi_{\delta,q}$
which solves (\ref{eq:P-Kort}) with
\[
\|\phi_{\delta,q}\|_{\mathcal{H}}\le C\varepsilon\left|\ln\varepsilon\right|.
\]
Moreover the map $q\mapsto\phi_{\delta,q}$ is a $C^{1}(\partial M,\mathcal{H})$
map.
\end{prop}
\begin{proof}
First we prove that the nonlinear operator $N$ defined (\ref{eq:defN})
is a contraction on a suitable ball of $\mathcal{H}$. Recalling that
\[
\|N(\phi_{1})-N(\phi_{2})\|_{\mathcal{H}}=\|N(\phi_{1})-N(\phi_{2})\|_{H}+|N(\phi_{1})-N(\phi_{2})|_{L^{s_{\varepsilon}}(\partial M)}
\]
we estimate the two right hand side terms separately.

By the continuity of $i^{*}:L^{\frac{2(n-1)}{n}}(\partial M)\rightarrow H$,
and by Lagrange theorem we have
\begin{multline*}
\|N(\phi_{1})-N(\phi_{2})\|_{H}\\
\le\left\Vert \left(f'_{\varepsilon}\left(W_{\delta,q}+\theta\phi_{1}+(1-\theta)\phi_{2}+\delta V_{\delta,q}\right)-f'_{\varepsilon}(W_{\delta,q}+\delta V_{\delta,q})\right)[\phi_{1}-\phi_{2}]\right\Vert _{L^{\frac{2(n-1)}{n}}(\partial M)}
\end{multline*}
and, since $|\phi_{1}-\phi_{2}|^{\frac{2(n-1)}{n}}\in L^{\frac{n}{n-2}}(\partial M)$
and $|f_{\varepsilon}'(\cdot)|^{\frac{2(n-1)}{n}}\in L^{\frac{n}{2}}(\partial M)$,
we have 
\begin{multline*}
\|N(\phi_{1})-N(\phi_{2})\|_{H}\\
\le\left\Vert \left(f'_{\varepsilon}\left(W_{\delta,q}+\theta\phi_{1}+(1-\theta)\phi_{2}+\delta V_{\delta,q}\right)-f'_{\varepsilon}(W_{\delta,q})+\delta V_{\delta,q}\right)\right\Vert _{L^{\frac{2(n-1)}{2}}(\partial M)}\|\phi_{1}-\phi_{2}\|_{H}\\
=\gamma\|\phi_{1}-\phi_{2}\|_{H}
\end{multline*}
where we can choose
\[
\gamma:=\left\Vert \left(f_{\varepsilon}'\left(W_{\delta,q}+\theta\phi_{1}+(1-\theta)\phi_{2}+\delta V_{\delta,q}\right)-f_{\varepsilon}'(W_{\delta,q}+\delta V_{\delta,q})\right)\right\Vert _{L^{\frac{2(n-1)}{n-2}}(\partial M)}<1,
\]
provided $\|\phi_{1}\|_{H}$ and $\|\phi_{2}\|_{H}$ sufficiently
small. 

For the second term we argue in a similar way and, recalling that,
by (\ref{eq:nittka}), $|i^{*}(g)|_{L^{s_{\varepsilon}}(\partial M)}\le|g|_{L^{\frac{2(n-1)+n(n-2)\varepsilon}{n+(n-2)\varepsilon}}(\partial M)}$,
we have
\begin{multline*}
|N(\phi_{1})-N(\phi_{2})|_{L^{s_{\varepsilon}}(\partial M)}\\
\le\left|\left(f'_{\varepsilon}\left(W_{\delta,q}+\theta\phi_{1}+(1-\theta)\phi_{2}+\delta V_{\delta,q}\right)-f'_{\varepsilon}(W_{\delta,q}+\delta V_{\delta,q})\right)[\phi_{1}-\phi_{2}]\right|_{L^{\frac{2(n-1)+n(n-2)\varepsilon}{n+(n-2)\varepsilon}}(\partial M)}
\end{multline*}
Since $\phi_{1},\phi_{2},W_{\delta,q}V_{\delta,q}\in L^{s_{\varepsilon}}$
we have that $|\phi_{1}-\phi_{2}|^{\frac{2(n-1)+n(n-2)\varepsilon}{n+(n-2)\varepsilon}}\in L^{\frac{n+(n-2)\varepsilon}{n-2}}(\partial M)$
and $|f'(\cdot)|^{\frac{2(n-1)+n(n-2)\varepsilon}{n+(n-2)\varepsilon}}\in L^{\frac{n+(n-2)\varepsilon}{2+(n-2)\varepsilon}}(\partial M)$.
So we conclude as above that we can choose $\left|\phi_{1}\right|_{L^{s_{\varepsilon}}(\partial M)}$,
$\left|\phi_{2}\right|_{L^{s_{\varepsilon}}(\partial M)}$ sufficiently
small in order to get
\[
|N(\phi_{1})-N(\phi_{2})|_{L^{s_{\varepsilon}}(\partial M)}\le\gamma\left|\phi_{1}-\phi_{2}\right|_{L^{s_{\varepsilon}}(\partial M)}.
\]
So 
\[
\|N(\phi_{1})-N(\phi_{2})\|_{\mathcal{H}}\le\gamma\|\phi_{1}-\phi_{2}\|_{\mathcal{H}}
\]
with $\gamma<1$, provided $\|\phi_{1}\|_{\mathcal{H}}$, $\|\phi_{2}\|_{\mathcal{H}}$
small enough.

With the same strategy it is possible to prove that if $\|\phi\|_{\mathcal{H}}$
is sufficiently small there exists $\bar{\gamma}<1$ such that $\|N(\phi)\|_{\mathcal{H}}\le\bar{\gamma}\|\phi\|_{\mathcal{H}}$. 

At this point, recalling Lemma \ref{lem:L} and Lemma \ref{lem:R},
it is not difficult to prove that there exists a constant $C>0$ such
that, if $\|\phi\|_{\mathcal{H}}\le C\varepsilon\left|\ln\varepsilon\right|$
then the map
\[
T(\phi):=L^{-1}(N(\phi)+R_{\varepsilon,\delta,q})
\]
is a contraction from the ball $\|\phi\|_{\mathcal{H}}\le C\varepsilon\left|\ln\varepsilon\right|$
in itself, and we get the first claim by the Contraction Mapping Theorem.
The regularity claim can be proven by means of the Implicit Function
Theorem.
\end{proof}

\section{The reduced problem}

For any given $(\delta,q)$, we are able to solve the infinite dimensional
problem (\ref{eq:P-Kort}) by Propostition \ref{prop:phi}. Now, set
$\delta=\lambda\sqrt{\varepsilon}$, we try to solve (\ref{eq:P})
finding a critical point of the functional
\begin{equation}
J_{\varepsilon}(u):=\frac{1}{2}\int_{M}|\nabla_{g}u|^{2}+au^{2}d\mu_{g}-\frac{(n-2)^{2}}{2(n-1)+\varepsilon(n-2)}\int_{\partial M}\left(u^{+}\right)^{\frac{2(n-1)}{n-2}+\varepsilon}d\sigma.\label{eq:Jeps}
\end{equation}
evaluated in $W_{\lambda\sqrt{\varepsilon},q}+\lambda\sqrt{\varepsilon}V_{\lambda\sqrt{\varepsilon},q}+\phi_{\lambda\sqrt{\varepsilon},q}$.
We observe that, by Remark \ref{rem:Nit}, the functional $J_{\varepsilon}$
is well defined on $\mathcal{H}$. Since $J_{\varepsilon}\left(W_{\lambda\sqrt{\varepsilon},q}+\lambda\sqrt{\varepsilon}V_{\lambda\sqrt{\varepsilon},q}+\phi_{\lambda\sqrt{\varepsilon},q}\right)$
depends only, given $\varepsilon$, on $(\lambda,q)\in[a,b]\times\partial M$,
we set $I_{\varepsilon}(\lambda,q):=J_{\varepsilon}\left(W_{\lambda\sqrt{\varepsilon},q}+\delta V_{\lambda\sqrt{\varepsilon},q}+\phi_{\lambda\sqrt{\varepsilon},q}\right)$.
\begin{lem}
\label{lem:JWpiuPhi}Assume $n\ge7$ and $\delta=\lambda\sqrt{\varepsilon}$.
It holds 
\[
\left|I_{\varepsilon}(\lambda,q)-J_{\varepsilon}\left(W_{\delta,q}+\delta V_{\delta,q}\right)\right|=o(\varepsilon)
\]
$C^{0}$-uniformly for $q\in\partial M$ and $\lambda$ in a compact
set of $(0,+\infty)$.
\end{lem}
This result can be obtained following the lines of \cite[Lemma 6]{GMP19},
so we postpone the proof in the Appendix, and we proceed to the main
result of this section.
\begin{prop}
\label{lem:expJeps}Assume $n\ge7$ and $\delta=\lambda\sqrt{\varepsilon}$.
It holds 
\[
J_{\varepsilon}(W_{\lambda\sqrt{\varepsilon},q}+\lambda\sqrt{\varepsilon}V_{\lambda\sqrt{\varepsilon},q})=A+B(\varepsilon)+\varepsilon\lambda\varphi(q)+C\varepsilon\ln\lambda+o(\varepsilon),
\]
$C^{0}$-uniformly for $q\in\partial M$ and $\lambda$ in a compact
set of $(0,+\infty)$, where

\begin{align*}
A= & \frac{(n-2)(n-3)}{2(n-1)^{2}}\omega_{n-1}I_{n-1}^{n}\\
B(\varepsilon)= & \varepsilon\left[\frac{(n-2)^{3}}{2(n-1)}\int_{\mathbb{R}^{n-1}}U^{\frac{2(n-1)}{n-2}}(z,0)dz-\frac{n-2}{2(n-1)}\int_{\mathbb{R}^{n-1}}U^{\frac{2(n-1)}{n-2}}(z,0)\ln U(z,0)dz\right]\\
 & -\varepsilon|\ln\varepsilon|\frac{(n-2)^{2}}{8(n-1)}\int_{\mathbb{R}^{n-1}}U^{\frac{2(n-1)}{n-2}}(z,0)dz\\
\varphi(q)= & \frac{1}{2}\int_{\mathbb{R}_{+}^{n}}\Delta v_{q}v_{q}dzdt-\frac{(n-6)(n-2)\omega_{n-1}I_{n-1}^{n}}{4(n-1)^{2}(n-4)}\|\pi(q)\|^{2}\le0.\\
C= & \frac{(n-2)^{2}(n-3)}{4(n-1)^{2}}\omega_{n-1}I_{n-1}^{n}>0.
\end{align*}
\end{prop}
\begin{proof}
Since $\frac{(n-2)^{2}}{2(n-1)+\varepsilon(n-2)}=\frac{(n-2)^{2}}{2(n-1)}-\varepsilon\frac{(n-2)^{3}}{2(n-1)}+o(\varepsilon)$,
we can write
\begin{multline*}
J_{\varepsilon}(W_{\delta,q}+\delta V_{\delta,q})\\
=\frac{1}{2}\int_{M}|\nabla_{g}(W_{\delta,q}+\delta V_{\delta,q})|^{2}+a(W_{\delta,q}+\delta V_{\delta,q})^{2}d\mu_{g}\\
-\left[\frac{(n-2)^{2}}{2(n-1)}-\varepsilon\frac{(n-2)^{3}}{2(n-1)}+o(\varepsilon)\right]\int_{\partial M}\left((W_{\delta,q}+\delta V_{\delta,q})^{+}\right)^{\frac{2(n-1)}{n-2}+\varepsilon}d\sigma\\
=\frac{1}{2}\int_{M}|\nabla_{g}(W_{\delta,q}+\delta V_{\delta,q})|^{2}+a(W_{\delta,q}+\delta V_{\delta,q})^{2}d\mu_{g}\\
-\frac{(n-2)^{2}}{2(n-1)}\int_{\partial M}\left((W_{\delta,q}+\delta V_{\delta,q})^{+}\right)^{\frac{2(n-1)}{n-2}}d\sigma\\
-\frac{(n-2)^{2}}{2(n-1)}\int_{\partial M}\left((W_{\delta,q}+\delta V_{\delta,q})^{+}\right)^{\frac{2(n-1)}{n-2}+\varepsilon}-\left((W_{\delta,q}+\delta V_{\delta,q})^{+}\right)^{\frac{2(n-1)}{n-2}}d\sigma\\
+\left[\varepsilon\frac{(n-2)^{3}}{2(n-1)}+o(\varepsilon)\right]\int_{\partial M}\left((W_{\delta,q}+\delta V_{\delta,q})^{+}\right)^{\frac{2(n-1)}{n-2}+\varepsilon}d\sigma.
\end{multline*}
For the first part we proceed as in \cite[Proposition 13]{GMP18}
(which we refer to for the proof) obtaining that
\begin{multline*}
\frac{1}{2}\int_{M}|\nabla_{g}(W_{\delta,q}+\delta V_{\delta,q})|^{2}+a(W_{\delta,q}+\delta V_{\delta,q})^{2}d\mu_{g}\\
-\frac{(n-2)^{2}}{2(n-1)}\int_{\partial M}\left((W_{\delta,q}+\delta V_{\delta,q})^{+}\right)^{\frac{2(n-1)}{n-2}}d\sigma=A+\varepsilon\lambda^{2}\varphi(q)+o(\varepsilon).
\end{multline*}
Using again (\ref{eq:Uallaeps}) and (\ref{eq:deltaallaeps}), proceeding
similarly to (\ref{eq:feps-f0}), and recalling that $\delta=\lambda\sqrt{\varepsilon}$
we have

\begin{multline*}
\int_{\partial M}\left((W_{\delta,q}+\delta V_{\delta,q})^{+}\right)^{\frac{2(n-1)}{n-2}+\varepsilon}-\left((W_{\delta,q}+\delta V_{\delta,q})^{+}\right)^{\frac{2(n-1)}{n-2}}d\sigma\\
=\int_{|z|<\frac{1}{\delta}}\frac{1}{\delta^{\varepsilon\frac{n-2}{2}}}\left((U+\delta v_{q})^{\varepsilon}-1\right)(U+\delta v_{q})^{\frac{2(n-1)}{n-2}}dz+o(\delta^{2})\\
\le\int_{\mathbb{R}^{n-1}}\left(-\frac{n-2}{4}\varepsilon\ln\varepsilon-\frac{n-2}{2}\varepsilon\ln\lambda+\varepsilon\ln(U)+O(\varepsilon^{2}\ln\varepsilon)\right)U^{\frac{2(n-1)}{n-2}}dz+o(\varepsilon)\\
=\frac{n-2}{4}\varepsilon|\ln\varepsilon|\int_{\mathbb{R}^{n-1}}U^{\frac{2(n-1)}{n-2}}dz+\varepsilon\int_{\mathbb{R}^{n-1}}U^{\frac{2(n-1)}{n-2}}\ln(U)dz\\
-\frac{n-2}{2}\varepsilon\ln\lambda\int_{\mathbb{R}^{n-1}}U^{\frac{2(n-1)}{n-2}}dz+o(\varepsilon).
\end{multline*}
Finally, with the same technique, 
\begin{multline*}
\left[\varepsilon\frac{(n-2)^{3}}{2(n-1)}+o(\varepsilon)\right]\int_{\partial M}\left((W_{\delta,q}+\delta V_{\delta,q})^{+}\right)^{\frac{2(n-1)}{n-2}+\varepsilon}d\sigma\\
=\left[\varepsilon\frac{(n-2)^{3}}{2(n-1)}+o(\varepsilon)\right]\int_{|z|<\frac{1}{\delta}}\frac{1}{\delta^{\varepsilon\frac{n-2}{2}}}(U+\delta v_{q})^{\frac{2(n-1)}{n-2}}(U+\delta v_{q})^{\varepsilon}dz+o(\delta^{2})\\
=\varepsilon\frac{(n-2)^{3}}{2(n-1)}\int_{\mathbb{R}^{n-1}}U^{\frac{2(n-1)}{n-2}}+o(\varepsilon),
\end{multline*}
and the proof follows easily taking in account that
\[
\frac{(n-2)^{2}}{4(n-1)}\int_{\mathbb{R}^{n-1}}U^{\frac{2(n-1)}{n-2}}dz=\frac{(n-2)^{2}}{4(n-1)}\omega_{n-1}I_{n-1}^{n-2}=\frac{(n-2)^{2}(n-3)}{4(n-1)^{2}}\omega_{n-1}I_{n-1}^{n}.
\]
\end{proof}

\section{Proof of Theorem \ref{almaraz}.}

Once we have a critical point of the reduced functional 
\[
I_{\varepsilon}(\lambda,q):=J_{\varepsilon}(W_{\lambda\sqrt{\varepsilon},q}+\lambda\sqrt{\varepsilon}V_{\lambda\sqrt{\varepsilon},q}+\phi_{\lambda\sqrt{\varepsilon},q})
\]
we solve Problem (\ref{eq:P}). In fact it holds the following result.
\begin{lem}
\label{lem:punticritici}If $(\bar{\lambda},\bar{q})\in(0,+\infty)\times\partial M$
is a critical point for the reduced functional $I_{\varepsilon}(\lambda,q)$,
then the function $W_{\bar{\lambda}\sqrt{\varepsilon},\bar{q}}+\bar{\lambda}\sqrt{\varepsilon}V_{\lambda\sqrt{\varepsilon},\bar{q}}+\phi_{\lambda\sqrt{\varepsilon},q}$
is a solution of (\ref{eq:P}). Here $\phi_{\lambda\sqrt{\varepsilon},q}$
is defined in Proposition \ref{prop:phi}.
\end{lem}
\begin{proof}
The proof is similar to the proofs of \cite[Lemma 15]{GMP18} and
\cite[Proposition 5]{GMP16}, so we sketch only the main steps. Set
$q=q(y)=\psi_{\bar{q}}^{\partial}(y)$. Since $(\bar{\lambda},\bar{q})$
is a critical point for the $I_{\varepsilon}(\lambda,q)$, and since
$W_{\bar{\lambda}\sqrt{\varepsilon},\bar{q}}+\bar{\lambda}\sqrt{\varepsilon}V_{\lambda\sqrt{\varepsilon},\bar{q}}+\phi_{\lambda\sqrt{\varepsilon},q}$
solves (\ref{eq:P-Kort}), we have, for $h=1,\dots,n-1$,

\begin{align*}
0= & \left.\frac{\partial}{\partial y_{h}}I_{\varepsilon}(\bar{\lambda},q(y))\right|_{y=0}\\
= & \langle\!\langle W_{\bar{\lambda}\sqrt{\varepsilon},q(y)}+\bar{\lambda}\sqrt{\varepsilon}V_{\bar{\lambda}\sqrt{\varepsilon},q(y)}+\phi_{\bar{\lambda}\sqrt{\varepsilon},q(y)}-i^{*}\left(f_{\varepsilon}(W_{\bar{\lambda}\sqrt{\varepsilon},q(y)}+\bar{\lambda}\sqrt{\varepsilon}V_{\bar{\lambda}\sqrt{\varepsilon},q(y)}+\phi_{\bar{\lambda}\sqrt{\varepsilon},q(y)})\right),\\
 & \left.\frac{\partial}{\partial y_{h}}(W_{\bar{\lambda}\sqrt{\varepsilon},q(y)}+\bar{\lambda}\sqrt{\varepsilon}V_{\bar{\lambda}\sqrt{\varepsilon},q(y)}+\phi_{\bar{\lambda}\sqrt{\varepsilon},q(y)})\rangle\!\rangle_{H}\right|_{y=0}\\
= & \sum_{i=1}^{n}c_{\varepsilon}^{i}\left.\langle\!\langle Z_{\bar{\lambda}\sqrt{\varepsilon},q(y)}^{i},\frac{\partial}{\partial y_{h}}(W_{\bar{\lambda}\sqrt{\varepsilon},q(y)}+\bar{\lambda}\sqrt{\varepsilon}V_{\bar{\lambda}\sqrt{\varepsilon},q(y)}+\phi_{\bar{\lambda}\sqrt{\varepsilon},q(y)})\rangle\!\rangle_{H}\right|_{y=0}\\
= & \sum_{i=1}^{n}c_{\varepsilon}^{i}\left.\langle\!\langle Z_{\bar{\lambda}\sqrt{\varepsilon},q(y)}^{i},\frac{\partial}{\partial y_{h}}W_{\varepsilon\bar{\lambda},q(y)}\rangle\!\rangle_{H}\right|_{y=0}+\bar{\lambda}\sqrt{\varepsilon}\sum_{i=1}^{n}c_{\varepsilon}^{i}\left.\langle\!\langle Z_{\varepsilon\bar{\lambda},q(y)}^{i},\frac{\partial}{\partial y_{h}}V_{\varepsilon\bar{\lambda},q(y)}\rangle\!\rangle_{H}\right|_{y=0}\\
 & -\sum_{i=1}^{n}c_{\varepsilon}^{i}\left.\langle\!\langle\frac{\partial}{\partial y_{h}}Z_{\bar{\lambda}\sqrt{\varepsilon},q(y)}^{i},\Phi_{\varepsilon\bar{\lambda},q(y)}\rangle\!\rangle_{H}\right|_{y=0}
\end{align*}

Arguing as in Lemma 6.1 and Lemma 6.2 of \cite{MP09} we have
\begin{align*}
\langle\!\langle Z_{\bar{\lambda}\sqrt{\varepsilon},q(y)}^{i},\frac{\partial}{\partial y_{h}}W_{\bar{\lambda}\sqrt{\varepsilon},q(y)})\rangle\!\rangle_{H}= & \frac{\delta_{ih}}{\bar{\lambda}\sqrt{\varepsilon}}+o(1)\\
\langle\!\langle Z_{\bar{\lambda}\sqrt{\varepsilon},q(y)}^{i},\frac{\partial}{\partial y_{h}}V_{\bar{\lambda}\sqrt{\varepsilon},q(y)}\rangle\!\rangle_{H} & \le\left\Vert Z_{\varepsilon\bar{\lambda},q(y)}^{i}\right\Vert _{H}\left\Vert \frac{\partial}{\partial y_{h}}V_{\varepsilon\bar{\lambda},q(y)}\right\Vert _{H}=O\left(\frac{1}{\sqrt{\varepsilon}}\right)\\
\langle\!\langle\frac{\partial}{\partial y_{h}}Z_{\varepsilon\bar{\lambda},q(y)}^{i},\Phi_{\varepsilon\bar{\lambda},q(y)}\rangle\!\rangle_{H} & \le\left\Vert \frac{\partial}{\partial y_{h}}Z_{\varepsilon\bar{\lambda},q(y)}^{i}\right\Vert _{H}\left\Vert \Phi_{\varepsilon\bar{\lambda},q(y)}\right\Vert _{H}=o(1).
\end{align*}
We conclude that 
\[
0=\frac{1}{\bar{\lambda}\sqrt{\varepsilon}}\sum_{i=1}^{n}c_{\varepsilon}^{i}\left(\delta_{ih}+O(1)\right)
\]
and so $c_{\varepsilon}^{h}=0$, where $h=1,\dots,n-1$.

Arguing analogously for $\left.\frac{\partial}{\partial\lambda}I_{\varepsilon}(\lambda,\bar{q})\right|_{\lambda=\bar{\lambda}}$
we can prove that $c_{\varepsilon}^{i}=0$ for all $i=1,\dots,n$,
obtaining that $W_{\bar{\lambda}\sqrt{\varepsilon},\bar{q}}+\bar{\lambda}\sqrt{\varepsilon}V_{\lambda\sqrt{\varepsilon},\bar{q}}+\phi_{\lambda\sqrt{\varepsilon},q}$
solves also (\ref{eq:P-K}), and so the proof is complete
\end{proof}
\begin{proof}[Proof of Theorem \ref{almaraz}]
 By our assumption on the second fundamental form and by (\ref{new}),
we have that the function $\varphi(q)$ defined in Proposition \ref{lem:expJeps}
is strictly negative on $\partial M$. We recall as well, that the
number $C$ defined in the same proposition is positive. Then, defined
\begin{align*}
I & :[a,b]\times\partial M\rightarrow\mathbb{R}\\
I(\lambda,q) & =\lambda\varphi(q)+C\ln d
\end{align*}
 we have that for any $M<0$ there exist $a,b$ such that 
\[
I(\lambda,q)<M\text{ for any }q\in\partial M,\ \lambda\not\in[a,b]
\]
and
\[
\frac{\partial I}{\partial\lambda}(a,q)\neq0,\ \ \frac{\partial I}{\partial\lambda}(a,q)\neq0\ \forall q\in\partial M.
\]
Then the function $I$ admits a absolute maximum on $[a,b]\times\partial M$.
This maximum is also $C^{0}$-stable. in other words, if $(\lambda_{0},q_{0})$
is the maximum point for $I$, for any function $f\in C^{1}([a,b]\times\partial M)$
with $\|f\|_{C^{0}}$ sufficiently small, then the function $I+f$
on $[a,b]\times\partial M$ admits a maximum point $(\bar{\lambda},\bar{q})$
close to $(\lambda_{0},q_{0})$. 

Then, taken an $\varepsilon$ sufficiently small, in light of Proposition
\ref{lem:JWpiuPhi} and Proposition \ref{lem:expJeps}, there exists
a pair $(\lambda_{\varepsilon},q_{\varepsilon})$ maximum point for
$I_{\varepsilon}(\lambda,q)$. Thus, by Lemma \ref{lem:punticritici},
$u_{\varepsilon}:=W_{\lambda_{\varepsilon}\sqrt{\varepsilon},q_{\varepsilon}}+\lambda_{\varepsilon}\sqrt{\varepsilon}V_{\lambda_{\varepsilon}\sqrt{\varepsilon},q_{\varepsilon}}+\phi_{\lambda_{\varepsilon}\sqrt{\varepsilon},q_{\varepsilon}}\in\mathcal{H}$
is a solution of (\ref{eq:Pmain}). By construction $u_{\varepsilon}$
blows up at $q_{\varepsilon}\rightarrow q_{0}$ when $\varepsilon\rightarrow0$.
\end{proof}

\section{Appendix}
\begin{proof}[Proof of Lemma \ref{lem:JWpiuPhi}]
 We have, for some $\theta\in(0,1)$ 
\begin{multline*}
J_{\varepsilon}(W_{\delta,q}+\delta V_{\delta,q}+\phi_{\delta,q})-J_{\varepsilon}(W_{\delta,q}+\delta V_{\delta,q})=J_{\varepsilon}'(W_{\delta,q}+\delta^{2}V_{\delta,q})[\phi_{\delta,q}]\\
+\frac{1}{2}J_{\varepsilon}''(W_{\delta,q}+\delta V_{\delta,q}+\theta\phi_{\delta,q})[\phi_{\delta,q},\phi_{\delta,q}]\\
=\int_{M}\left(\nabla_{g}W_{\delta,q}+\delta\nabla_{g}V_{\delta,q}\right)\nabla_{g}\phi_{\delta,q}+a(x)\left(W_{\delta,q}+\delta V_{\delta,q}\right)\phi_{\delta,q}d\mu_{g}\\
-(n-2)\int_{\partial M}\left(\left(W_{\delta,q}+\delta^{2}V_{\delta,q}\right)^{+}\right)^{\frac{n}{n-2}+\varepsilon}\phi_{\delta,q}d\sigma_{g}\\
+\frac{1}{2}\int_{M}|\nabla_{g}\phi_{\delta,q}|^{2}+a(x)\phi_{\delta,q}^{2}d\mu_{g}\\
-\frac{n+\varepsilon(n-2)}{2}\int_{\partial M}\Lambda_{q}^{\varepsilon}\left(\left(W_{\delta,q}+\delta V_{\delta,q}+\theta\phi_{\delta,q}\right)^{+}\right)^{\frac{2}{n-2}+\varepsilon}\phi_{\delta,q}^{2}d\sigma_{g}.
\end{multline*}
By definition of $\|\cdot\|_{H}$,
\[
\int_{M}|\nabla_{g}\phi_{\delta,q}|^{2}+a\phi_{\delta,q}^{2}d\mu_{g}=\|\phi_{\delta,q}\|_{H}^{2}=o(\varepsilon).
\]
By Holder inequality one can easily obtain
\[
\int_{M}a\left(W_{\delta,q}+\delta V_{\delta,q}\right)\phi_{\delta,q}d\mu_{g}\le C|W_{\delta,q}|_{L^{\frac{2n}{n+2}}(M)}|\phi_{\delta,q}|_{L^{\frac{2n}{n-2}}(M)}=O(\delta^{2})\|\phi_{\delta,q}\|_{H}=o(\varepsilon)
\]
Since $\left(\left(W_{\delta,q}+\delta V_{\delta,q}+\theta\phi_{\delta,q}\right)^{+}\right)^{\frac{2}{n-2}+\varepsilon}$
belongs to $L_{\tilde{g}}^{\frac{2(n-1)+n(n-2)\varepsilon}{2+(n-2)\varepsilon}}$and
since $2\left(\frac{2(n-1)+n(n-2)\varepsilon}{2+(n-2)\varepsilon}\right)'=\frac{4(n-1)+2n(n-2)\varepsilon}{2(n-2)+(n-1)(n-2)\varepsilon}<s_{\varepsilon}$,
by Holder inequality we obtain
\begin{multline*}
\int_{\partial M}\left(\left(W_{\delta,q}+\delta V_{\delta,q}+\theta\phi_{\delta,q}\right)^{+}\right)^{\frac{2}{n-2}+\varepsilon}\phi_{\delta,q}^{2}d\sigma_{g}\\
\le C\left(\left|W_{\delta,q}+\delta V_{\delta,q}+\theta\phi_{\delta,q}\right|_{L^{s_{\varepsilon}}(\partial M)}^{\frac{2}{n-2}}\right)\|\phi_{\delta,q}\|_{H}^{2}=o(\varepsilon).
\end{multline*}
By integration by parts we have 
\begin{multline}
\int_{M}\left(\nabla_{g}W_{\delta,q}+\delta\nabla_{g}V_{\delta,q}\right)\nabla_{g}\phi_{\delta,q}d\mu_{g}=-\int_{M}\Delta_{g}\left(W_{\delta,q}+\delta V_{\delta,q}\right)\phi_{\delta,q}d\mu_{g}\\
+\int_{\partial M}\left(\frac{\partial}{\partial\nu}W_{\delta,q}+\delta\frac{\partial}{\partial\nu}V_{\delta,q}\right)\phi_{\delta,q}d\sigma_{g}.\label{eq:parts}
\end{multline}
and, as in (\ref{eq:deltaw+v}), we get 
\begin{multline*}
\int_{M}\Delta_{g}\left(W_{\delta,q}+\delta V_{\delta,q}\right)\phi_{\delta,q}d\mu_{g}\\
\le|\Delta_{g}(W_{\delta,q}+\delta V_{\delta,q})|_{L^{\frac{2n}{n+2}}(M)}\|\phi_{\delta,q}\|_{H}=O(\delta^{2})\|\phi_{\delta,q}\|_{H}=o(\varepsilon),
\end{multline*}
and for the boundary term in (\ref{eq:parts}), in light of (\ref{eq:I2})
and the following formulas, we get
\begin{multline*}
\int_{\partial M}\left[\left(\frac{\partial}{\partial\nu}W_{\delta,q}+\delta\frac{\partial}{\partial\nu}V_{\delta,q}\right)-(n-2)\left(\left(W_{\delta,q}+\delta V_{\delta,q}\right)^{+}\right)^{\frac{n}{n-2}}\right]\phi_{\delta,q}d\sigma_{g}\\
=\left|(n-2)\left(\left(W_{\delta,q}+\delta V_{\delta,q}\right)^{+}\right)^{\frac{n}{n-2}}-\frac{\partial}{\partial\nu}W_{\delta,q}\right|_{L_{g}^{\frac{2(n-1)}{n}}(\partial M)}|\phi_{\delta,q}|_{L_{g}^{\frac{2(n-1)}{n-2}}(M)}\\
=O(\delta^{2})\|\phi_{\delta,q}\|_{H}=o(\varepsilon).
\end{multline*}
At this point it remains to estimate 
\[
\int_{\partial M}\left[\left(\left(W_{\delta,q}+\delta V_{\delta,q}\right)^{+}\right)^{\frac{n}{n-2}+\varepsilon}-\left(\left(W_{\delta,q}+\delta V_{\delta,q}\right)^{+}\right)^{\frac{n}{n-2}}\right]\phi_{\delta,q}d\sigma_{g}
\]
and we proceed as in (\ref{eq:feps-f0}) to get
\begin{multline*}
\int_{\partial M}\left[\left(\left(W_{\delta,q}+\delta V_{\delta,q}\right)^{+}\right)^{\frac{n}{n-2}+\varepsilon}-\left(\left(W_{\delta,q}+\delta V_{\delta,q}\right)^{+}\right)^{\frac{n}{n-2}}\right]\phi_{\delta,q}d\sigma_{g}\\
\le\left|\left(\left(W_{\delta,q}+\delta V_{\delta,q}\right)^{+}\right)^{\frac{n}{n-2}+\varepsilon}-\left(\left(W_{\delta,q}+\delta V_{\delta,q}\right)^{+}\right)^{\frac{n}{n-2}}\right|_{L^{\frac{2(n-1)}{n}}(\partial M)}\|\phi_{\delta,q}\|_{H}=o(\varepsilon),
\end{multline*}
and we conclude the proof.
\end{proof}

\end{document}